 \documentclass[final,1p,times]{elsarticle}

\RequirePackage{epsfig}

\usepackage{amssymb}
 \usepackage{amsthm}

\newtheorem{theorem}{Theorem}[section]

\newtheorem{prop}[theorem]{Proposition}

\newtheorem{cor}[theorem]{Corollary}

\newcommand{\CC}{{\mathbb C}}
\newcommand{\RR}{{\mathbb R}}
\newcommand{\NN}{{\mathbb N}}

\newcommand{\LL}{\mathcal L}
\newcommand{\GG}{\mathcal G}

\newcommand{\CMF}{\mathcal{CMF}}
\newcommand{\BF}{\mathcal{BF}}
\newcommand{\SF}{\mathcal{SF}}
\newcommand{\CBF}{\mathcal{CBF}}

\newcommand{\e}{\varepsilon}
\newcommand{\De}{\Delta}
\newcommand{\g}{\gamma}

\newcommand{\ph}{\phi}
\newcommand{\Ga}{\Gamma}
\newcommand{\Om}{\Omega}

\newcommand{\be}{\beta}
\newcommand{\al}{\alpha}
\newcommand{\la}{\lambda}

\newcommand{\Si}{\Sigma}
\newcommand{\phh}{\varphi}

\begin{document}

\begin{frontmatter}

\title{Subordination approach to multi-term time-fractional diffusion-wave equations}

\author{Emilia Bazhlekova 
\fnref{label1}} \ead{e.bazhlekova@math.bas.bg}
\fntext[label1]{Corresponding author}

\author{Ivan Bazhlekov}
\address{Institute of Mathematics and Informatics, Bulgarian Academy of Sciences, Acad. G. Bonchev Str.,
Bl. 8, Sofia 1113, Bulgaria}
 \ead{i.bazhlekov@math.bas.bg}

\begin{abstract}
 This paper is concerned with the fractional evolution equation with a discrete distribution of Caputo time-derivatives such that the largest and the smallest orders, $\al$ and $\al_m$,  satisfy the conditions $1<\al\le 2$ and $\al-\al_m\le 1$.
First, based on a study of the related propagation function, the nonnegativity of the fundamental solutions to the spatially one-dimensional Cauchy and signaling problems is proven and propagation speed of a disturbance is discussed. Next, we study the equation with a general linear spatial differential operator defined in a Banach space and suppose it generates a cosine family. A subordination principle is established, which implies the existence of a unique solution and gives an integral representation of the solution operator in terms
of the corresponding cosine family and a probability density function. Explicit representation of the probability density function is derived. The subordination principle is applied for obtaining regularity results. The analytical findings are supported by numerical work.
  
\end{abstract}


\end{frontmatter}

\section{Introduction}

A variety of generalized wave equations has been proposed to model wave propagation in complex media. One of them is the diffusion-wave equation with the Caputo fractional time derivative
\begin{equation}\label{DW}
 D_t^\al u(\mathbf{x},t)=\De_{\mathbf{x}} u (\mathbf{x},t), \ \ \al\in(1,2),
\end{equation}
which describes evolution processes intermediate between diffusion and wave propagation \cite{MainardiAML, MainardiFDW, M, A, Luchkopropspeed, LuchkoMainardi, sub2}. 
It is shown in \cite{MainardiFDW} that the spatially one-dimensional version of Eq.~(\ref{DW}) governs the propagation of mechanical diffusive waves in viscoelastic media exhibiting a power-law creep. Such waves are of relevance in acoustics, seismology, medical imaging, etc. For example, experimental evidence reveals that in a complex inhomogeneous conducting medium sound waves exhibit power-law attenuation (see, e.g., \cite{Duck} for applications to medical ultrasound).  In general, fractional derivatives in time reflect hereditary mechanisms of power-law type in diffusion or wave phenomena, see \cite{Luchkopropspeed} and the references cited there. 

In the attempt to find more adequate models, 
the single fractional time derivative in Eq.~(\ref{DW}) is often replaced by a discrete or continuous distribution of fractional time derivatives over the interval $(0,2]$, see e.g. \cite{A}, Chapter 6, and \cite{A1, A2, Trifce}. Concerning the spatial operator, besides the Laplacian, more general classes of operators have been also considered: the fractional Laplacian \cite{FractLap, FractLap1}, second order symmetric uniformly elliptic operators \cite{SaYa},  etc. To cover different spatial operators, a general unbounded linear operator on a Banach space of functions is taken in, e.g., \cite{sub2, A1, A2, B, KosticFCAA}. 

Assume $A$ is a general linear closed 
operator densely defined in a Banach space $X$ and 
 consider the fractional evolution equation 
\begin{equation}\label{st}
 D_t^{\al}u(\mathbf{x},t)=Au(\mathbf{x},t), \ \ t>0,\mathbf{x}\in\RR^n;\ \ \  \ u(\mathbf{x},0)=v(\mathbf{x})\in X, \ \ u_t(\mathbf{x},0)=0,\  \ 
\end{equation} 
where $\al\in(0,2]$ (the second initial condition is assumed only if $\al>1$). Problem (\ref{st}) is extensively studied, see e.g. \cite{B} for some basic definitions and properties. Denote by $S_\al(t)$ the solution operator corresponding to problem (\ref{st}). In the limiting case $\al=2$ the solution operator  $S_2(t)$ is the strongly continuous cosine function generated by the operator $A$
 (\cite{Arendt}, Section 3.14). 

The following subordination principle holds true \cite{B, Subordination0}: If problem (\ref{st}) is well posed for $\al=2$ then it is well posed for all $\al\in (0,2)$ and the solution operators of these two problems are related by the identity
\begin{equation}\label{subM}
S_\al(t)=t^{-{\al/2}}\int_0^\infty \Phi_{\al/2}(\tau t^{-\al/2})S_2(\tau)\, \mbox{d}\tau, \ \ t>0,
\end{equation}
where $\Phi_\be(z)$, $\be\in(0,1)$, is a function of the Wright type, also known as Mainardi function
\begin{equation}\label{Mainardi}
\Phi_\be(z)=\sum_{k=0}^\infty \frac{(-z)^k}{k!\Ga(-\be k+1-\be)},\ \ \be\in(0,1).
\end{equation}
It is a unilateral probability density function (p.d.f.) in the sense that
$$
\Phi_\be(t)\ge 0,\ t>0;\ \ \ \ \int_0^\infty \Phi_\be(t)\, \mbox{d}t=1.
$$
Let us note that function (\ref{Mainardi}) appears also in the fundamental solutions of the spatially one-dimensional Cauchy and signaling problems for equation (\ref{DW}), cf. \cite{MainardiAML, Luchkopropspeed, LuchkoMainardi}. 

By means of the subordination principle it is possible to construct new solutions from given ones 
as well as to study their regularity and asymptotic behavior.

Subordination principle in a general setting of abstract Volterra equations is introduced in \cite{Pruss}, Chapter~4. 
In the context of fractional evolution equations this principle has been applied for asymptotic analysis of fractional diffusion-wave equations \cite{sub1}, regularity and representation of solution of fractional diffusion equations in terms of integrated semigroups \cite{KeLiWa}, inverse problems \cite{sub0}, study of semilinear equations \cite{APLizama}, etc. Generalizations of this principle include distributed order diffusion equations \cite{K, Meer, ITSF, Mijena} and regularized resolvent families \cite{Abadias}. Let us note that the principle of subordination is closely related to the concept of subordination in stochastic processes  \cite{Feller}. For instance, the subordination results concerning fractional evolution equations in \cite{sub2, Meer, Mijena} are presented in such a context.
 
In this paper we are concerned with the following multi-term generalization of the fractional evolution equation (\ref{st})
\begin{equation}\label{A}
cD_t^{\al}u(\mathbf{x},t)+\sum_{j=1}^m {c_j} D_t^{\al_j}u(\mathbf{x},t)=Au(\mathbf{x},t), \ \ \ u(\mathbf{x},0)=v(\mathbf{x})\in X, \ u_t(\mathbf{x},0)=0,
\end{equation}
where the operator $A$ is a generator of a strongly continuous cosine function. We suppose that the parameters $\al,\al_j,c,c_j,$ satisfy the following restrictions 
\begin{equation}\label{al}
\al\in(1,2], \  \al>\al_1...>\al_m>0, \  \al-\al_m\le 1, \ c,c_j> 0, \ j=1,...,m,\ m\ge 0.  
\end{equation}

The two-term case of problem (\ref{A}), also referred to as time-fractional telegraph equation, 
is studied in several works. For a discussion on the applications and derivation of solutions we refer to \cite{At2term} and \cite{A}, Chapter 6, as well as to the works \cite{OB, Mamchuev}, \cite{Cattaneo}, and \cite{damping, Qi}, which consider the particular cases $\al=2\al_1\in(1,2)$; $\al=2,\al_1\in(1,2)$ and $\al\in(1,2), \al_1=1$, respectively. For the general case of multi-term time-fractional diffusion-wave equation see e.g. \cite{JLTB} where an analytical solution is derived, and \cite{DehghanJCAM} where a numerical approach for this equation is developed. In the papers \cite{L13} and \cite{KLW10} a two-term
time fractional differential equation is studied in the abstract setting. Abstract framework for the study of the general multi-term case is developed in \cite{ A1, A2, KosticFCAA}. In \cite{APLizama}, a semilinear generalization of equation (\ref{A}) is studied. 

We prove the following subordination identity for the solution operator $S(t)$ of problem (\ref{A}) for the  multi-term fractional evolution equation
\begin{equation}\label{sub0}
S(t)=\int_0^\infty \ph(t,\tau)S_2(\tau)\, \mbox{d}\tau, \ \ t>0,
\end{equation}
where $S_2(t)$ is
the cosine family generated by the operator $A$ and $\ph(t,\tau)$ is a probability density function in $\tau$, that is for any $t,\tau>0$
\begin{equation}\label{PD}
\ph (t,\tau)\ge 0,\ \ \int_0^\infty \ph (t,\tau)\, \mbox{d}\tau=1. 
\end{equation}
In fact the function $\ph(t,\tau)$ is related to the fundamental solution $\GG_c(x,t)$ of the Cauchy problem for the spatially one-dimensional version of equation (\ref{A}) 
as follows:
\begin{equation}\label{Gphi}
\GG_c(x,t)=\frac12\ph(t,|x|)\ \ x\in\RR.
\end{equation}

For $\al<2$ we prove that $\ph(t,\tau)$ admits an analytic extension to a sector in the complex plane $|\arg t|<\theta_0$. This together with (\ref{sub0}) implies the analyticity of the solution operator $S(t)$ of problem (\ref{A}) in the same sector. In particular, the established analyticity property infers infinite propagation speed  for $\al<2$  as in the single-term case. 

On the other hand, if $\al=2$,  the function $\ph(t,\tau)$ vanishes for $x>t/\sqrt{c}$. This means that the propagation speed is finite ($=1/\sqrt{c}$) and therefore the integral in (\ref{sub0}) is also finite. 

Further, an explicit integral representation for the p.d.f. $\ph(t,\tau)$ is derived, which makes formula (\ref{sub0}) appropriate for computation of the solution of problem (\ref{A}) from the cosine family generated by the operator $A$. 

Our proofs use essentially some facts from the theory of Bernstein functions \cite{CMF}. We refer also to \cite{GLS} for a useful selection of definitions, properties and their application to the proof of positivity of the fundamental solution to distributed-order diffusion-wave equations, as well as to the recent work of the authors \cite{CAMWA}, where this technique is also exploited.

The rest of this paper is organized as follows. 
Section 2 is concerned with the propagation function for problem (\ref{A}), its properties and explicit integral representation. Subordination principle to cosine families is derived in Section 3 and applied to prove regularity of the solution. In Section 4 a stronger subordination result to $S_\al(t)$ is briefly presented. Section 5 contains examples. 
A selection of definitions and properties concerning Bernstein functions and related classes of functions is given in an Appendix.

\section{Propagation function
}

Consider first the following problem for the spatially one-dimensional version of the multi-term equation in (\ref{A})
\begin{eqnarray}
&&cD_t^{\al}w(x,t)+\sum_{j=1}^m {c_j} D_t^{\al_j}w(x,t)=w_{xx}(x,t),\ \ x,t>0,\label{S}\\
&&w(x,0)=w_t(x,0)=0,\ \ x>0, \label{ic}\\
&& w(0,t)=H(t),\ \ w\to 0 \ {\mathrm as}\  x\to\infty,\ \ t>0. \ \ \ \label{bc}
\end{eqnarray}
Here $D_t^{\be},\ \be>0,$ denotes time-derivative in the Caputo sense, the parameters $\al,\al_j,c,c_j, j=1,...,m,$ satisfy conditions (\ref{al}), and $H(t)$ is the Heaviside unit step function. 

The solution $w(x,t)$ of problem (\ref{S})-(\ref{ic})-(\ref{bc}) is referred to as propagation function (cf. \cite{Pruss}, Section 4.5), since it represents the propagation in time of a disturbance at $x=0$.

We find the propagation function by the use of Laplace transform $$\mathcal{L}\{f(t)\}(s)=\widehat{f}(s)=\int_0^\infty e^{-st} f(t)\, \mbox{d}t$$
employing the fundamental formula for Caputo derivatives 
\begin{equation}\label{LapD} 
\mathcal{L}\{D_t^\be f\}(s)=s^\be \widehat{f}(s)-\sum_{k=0}^{n-1}f^{(k)}(0)s^{\be-1-k},\ \ n-1<\be\le n, \ n\in\NN.
\end{equation}

By applying Laplace transform with respect to the temporal variable in (\ref{S}) and (\ref{bc}) and taking into account initial conditions (\ref{ic}) we obtain using (\ref{LapD})
the following problem 
\begin{equation}\label{gs0}
g(s)\widehat{w}(x,s)=\widehat{w}_{xx}(x,s), \ \widehat{w}(0,s)=1/s,\ \widehat{w}(x,s)\to 0 \ {\mathrm as}\  x\to\infty, 
\end{equation}
where
\begin{equation}\label{gs}
g(s)=cs^{\al}+\sum_{j=1}^m {c_j} s^{\al_j}, \ \ s>0, 
\end{equation}
with parameters $\al,\al_j,c,c_j,j=1,...,m,$ satisfying ({\ref{al}).
Here $\widehat{w}(x,s)$ denotes the Laplace transform of the function $w(x,t)$ with respect to $t$.
Solving problem (\ref{gs0}) we deduce
\begin{equation}\label{soll}
\widehat{w}(x,s)=\frac{1}{s}\exp\left(-x\sqrt{g(s)}\right). 
\end{equation}

\subsection{Properties}

For our considerations it is essential that the propagation function $w(x,t)$ is nonnegative. This is fulfilled provided $\sqrt{g(s)}$ is a Bernstein function and the proof uses a standard argument based on the theory of Bernstein functions and related classes of functions (see the proof of Theorem \ref{33}). For definitions and a selection of properties of Bernstein functions ($\BF$), completely monotone functions ($\CMF$), complete Bernstein functions ($\CBF$) and Stieltjes functions ($\SF$)  see Appendix. 
In fact, we prove next a stronger property: $\sqrt{g(s)}\in \CBF\subset \BF$.

\begin{prop}\label{pg}
Assume $g(s)$ is defined by (\ref{gs}) with parameters $\al,\al_j,c,c_j,j=1,...,m,$ satisfying conditions (\ref{al}). Then $\sqrt{g(s)}$ is a complete Bernstein function.
\end{prop}
\begin{proof}
Consider first the case $\al_m\ge 1$. Set $f(s)=g(s)/s$. Since $2>\al,\al_j\ge 1$ the function $f(s)=cs^{\al-1}+\sum_{j=1}^m {c_j} s^{\al_j-1}\in\CBF$ as a sum of complete Bernstein functions.   Also, $s\in\CBF$. Then, applying property~(F) from the Appendix with $p=q=1/2$ it follows that $\sqrt{g(s)}=\sqrt{s}\sqrt{f(s)}\in\CBF.$

In the case $\al_m< 1$ we set $f(s)=g(s)/s^{\al_m}$. The assumption $0<\al-\al_m\le 1$ implies again $f(s)\in \CBF$. Since also $s^{\al_m}\in\CBF$, we obtain in the same way as above $\sqrt{g(s)}=\sqrt{s^{\al_m}}\sqrt{f(s)}\in\CBF$. 
\end{proof}

Let us note that constraints (\ref{al}) on the parameters of the problem are essential for the proof of Proposition~\ref{pg} and, therefore, for deriving most of the results in the present work. To illustrate this, suppose that the restriction on the distance between the largest and the smallest order of fractional derivative is violated, i.e. $\al-\al_m>1$. 

 Consider a simple two-term equation for which $g(s)=s^\al+s^{\al_1}$.  
If $\al-\al_1>1$ then representation $g(s)=s^{\al_1}(s^{\al-\al_1}+1)$ implies that there exists $s_0\in \CC\backslash (-\infty,0]$ such that $g(s_0)=0$. Therefore, ${g(s)}^{1/2}$ has a branch point in $\CC\backslash (-\infty,0]$ and, according to property~(G) in the Appendix, ${g(s)}^{1/2}\notin \CBF$. 

Also, the weaker property $\sqrt{g(s)}\in \BF$, which is sufficient for the proof of Theorem~\ref{33}, does not hold without a restriction on the distance $\al-\al_1$. Considering the above two-term example with different values of the parameters $\al$ and $\al_1$ such that $\al-\al_1>1$ (e.g. $\al=1.9,\al_1\in(0,0.5]$; $\al=1.8,\al_1\in(0,0.3]$) we obtain by direct computation that the second derivative $d^2/ds^2(\sqrt{g(s)})$ admits positive values for some $s>0$. Therefore, the function $\sqrt{g(s)}$ is not concave for all $s>0$, which implies that $\sqrt{g(s)}\notin \BF$. 

Proposition~\ref{pg} 
implies  important properties (\ref{propertiesw}) of the propagation function. 
To the best of the authors' knowledge, it is an open problem whether condition $\al-\al_m\le 1$ is necessary for these properties. 

\begin{theorem}\label{33}
The propagation function $w(x,t)$ satisfies the properties
\begin{equation}\label{propertiesw}
w(x,t)\ge 0,\ \ \ 
w_t(x,t)\ge 0,\ \ \  -
w_x(x,t)\ge 0,\ \ \ \ x,t >0.
\end{equation}
\end{theorem}
\begin{proof}
According to Bernstein's theorem it is sufficient to prove that the Laplace transforms of the three functions in (\ref{propertiesw}) are completely monotone. We have from Proposition~\ref{pg} that $\sqrt {g(s)}\in\BF$.
  Then, by property~(C) in the Appendix, the function $\exp\left(-x \sqrt{g(s)}\right)\in\CMF$ as a composition of the completely monotone exponential function and the Bernstein function  $\sqrt {g(s)}$.
	 Since $1/s\in\CMF$, by property~(B) in the Appendix also $\sqrt{g(s)}/s\in\CMF$. Then 
$\widehat{w}(x,s)= \frac{1}{s}\exp\left(-x \sqrt{g(s)}\right)\in \CMF$ as well as
\begin{equation}\label{dw}
\LL\{-w_x\}(x,s)=-\frac{\partial}{\partial x}\widehat{w}(x,s)=\frac{\sqrt{g(s)}}{s}\exp\left(-x \sqrt{g(s)}\right)\in\CMF
\end{equation} 
as products of two completely monotone functions. 

Further, (\ref{soll}) and (\ref{gs}) imply $\lim_{t\to 0}w(x,t)=\lim_{s\to \infty}s\widehat{w}(x,s)=0$, thus
$$\LL\{w_t\}(x,s)=s\widehat{w}(x,s)-w(x,0)=\exp\left(-x \sqrt{g(s)}\right)\in\CMF.$$
 \end{proof}

Theorem \ref{33} implies that $w(x,t)$ is a nonincreasing function in $x$ and nondecreasing function in $t$ with limiting values found by applying Tauberian theorems:
\begin{equation}\label{limits}
\lim_{t\to 0}w(x,t)=\lim_{s\to \infty}s\widehat{w}(x,s)=0,\ \ \ 
\lim_{t\to +\infty}w(x,t)=\lim_{s\to 0}s\widehat{w}(x,s)=1.
\end{equation}

Denote by $\GG_c(x,t)$ and $\GG_s(x,t)$ the fundamental solutions of the Cauchy and signaling problems for equation (\ref{S}) (for definitions see e.g. \cite{MainardiAML, LuchkoMainardi}). The fundamental solutions can be expressed in terms of the propagation function $w(x,t)$ as follows:
\begin{equation}\label{CS}
\GG_c(x,t)=-\frac{1}{2}w_x(|x|,t),\ \ x\in\RR;\ \ \ \GG_s(x,t)=w_t(x,t),\ \ x >0,
\end{equation}
and therefore Theorem \ref{33}  implies that they are nonnegative functions. Moreover, it follows from (\ref{CS}), (\ref{limits}) and (\ref{bc}) that 
\begin{equation}\label{Gpdf}
\int_{-\infty}^\infty\GG_c(x,t) \mbox{d}x=1,\ t>0;\ \ \ \int_{0}^\infty\GG_s(x,t) \mbox{d}t=1,\ x>0.
\end{equation}

Further properties of the fundamental solutions reflect those of the propagation function. 

For $\al<2$ the propagation function $w(x,t)$ (and hence also the fundamental solutions) admit an analytic extension to a sector in the complex plane $t\in\CC\backslash 0$, $|\arg t|<\theta_0$ (the proof is essentially the same as that of Theorem \ref{an}). Therefore, for any $x>0$ the set of zeros of $w(x,t)$ on $t>0$ can be only discrete. This together with  (\ref{propertiesw}) and (\ref{limits}) implies that $w(x,t)>0$ for all $x,t>0$ and a disturbance spreads infinitely fast. 

\begin{theorem}\label{infinite}
If $1<\al<2$ then $w(x,t)>0$ for all $x,t>0$. 
\end{theorem}

On the other hand, in the case $\al=2$, a disturbance spreads with finite speed as in the classical wave and telegraph equations. However, in contrast to the classical equations, in the case when at least one more time-derivative in equation (\ref{S}) is present, which is of noninteger order, a phenomenon of coexistence of finite propagation speed and absence of wave front is established. This is a memory effect, not observed in linear integer-order differential equations (see \cite{Pruss} for a discussion in the general case of Volterra equations). 

We will prove that 
for $\al=2$ there is a finite propagation speed $1/\sqrt{c}$. 
Define the function
$$h(s)=\sqrt{g(s)}-\sqrt{c} s.$$  Then (\ref{soll}) implies 
\begin{equation}\label{propspeed}
w(x,t)=\LL^{-1}\left\{\frac{1}{s}\exp\left(-xh(s)\right)\exp\left(-x\sqrt{c} s\right)\right\}
=w_0(x,t-\sqrt{c}x)H(t-\sqrt{c} x),
\end{equation}
where
$$w_0(x,t)=\LL^{-1}\left\{\frac{1}{s}\exp\left(-xh(s)\right)\right\}.
$$ 
Here we have used the property $\LL\{f(t-a)H(t-a)\}(s)=\exp(-as)\LL\{f\}(s)$.
Since $h(s), h'(s)\ge 0$ for $s>0$ and $\sqrt{g(s)}\in\BF$, it follows that $h(s)\in\BF$. Therefore, $w_0(x,t)\ge 0$ by the same argument as in the proof of Theorem \ref{33}. Formula (\ref{propspeed}) implies that the propagation function $w(x,t)$ vanishes for $x>t/\sqrt{c}$, i.e. the propagation speed is $1/\sqrt{c}$.

\begin{theorem}\label{w0}
If $\al=2$ then $w(x,t)\equiv 0$ for $x>t/\sqrt{c}$. 
\end{theorem}

Except in the two classical cases of wave equation ($m=0, \al=2$) and classical telegraph equation ($m=1, \al=2,\al_1=1$), $\lim_{s\to \infty}h(s)=\infty$, which implies that a wave front (jump discontinuity) at $x=t/\sqrt{c}$ is not present (cf. \cite{Pruss}, Chapter 5).

The behaviour of the propagation function $w(x,t)$ is illustrated in Figs.~1-3. Three different cases for the two-term equation are considered: the classical telegraph equation (Fig.~1) which exhibits finite propagation speed and wave front, an equation with $\al=2$ and $\al_1\in (1,2)$ (Fig.~2) exhibiting finite propagation speed and absence of wave front, and an equation with $\al<2$ (Fig.~3) exhibiting infinite propagation speed. Plots are obtained by numerical computation based on the explicit integral representation for $w(x,t)$ derived next. 

The numerical computations for producing all plots in this work are performed with MATLAB. For the numerical calculation of the improper integrals in (\ref{reshenie}), (\ref{sf}) and (\ref{Sncos}) the MATLAB function ``integral'' is used. 

\subsection{Explicit representation}
Let us first note that for multivalued functions in $\CC$ such as $s^\al=\exp(\al\ln s)$ always the principal branch is considered in this work.

Applying 
 the complex Laplace inversion formula to (\ref{soll}) yields: 
\begin{eqnarray}
w(x,t)&=&\frac{1}{2\pi\mathrm{i}}\int_{\g-\mathrm{i}\infty}^{\g+\mathrm{i}\infty}e^{st}\widehat{w}(x,s)\,\mathrm{d}s\nonumber\\
&=&\frac{1}{2\pi\mathrm{i}}\int_{\g-\mathrm{i}\infty}^{\g+\mathrm{i}\infty}\exp\left(st-x\sqrt{g(s)}\right)\,\frac{\mathrm{d}s}{s},\ \ \g>0.\label{in}
\end{eqnarray}

Since $\sqrt{g(s)}\in\CBF$ it can be analytically extended to $\CC\backslash (-\infty,0]$. Therefore, this holds also for the function under the integral sign in (\ref{in}). By the Cauchy's theorem, the integration on the contour $\{s=\g+\mathrm{i}r,\ r\in(-\infty,+\infty)\}$ 
can be replaced
by integration on the contour 
$D_R^-\cup D\cup D_0\cup D_R^+$, where (with appropriate orientation)
$$D=\{s=\mathrm{i}r,\ r\in(-\infty,-\e)\cup (\e,\infty)\},\ 
D_\e=\{ s=\e e^{\mathrm{i}\theta}, \ \theta\in[-\pi/2,\pi/2]\},$$  
$$
D_R^+=\{|s|=R, \ \Re s\in [0,\g], \ \Im s>0\},\ \ D_R^-=\{|s|=R, \ \Re s\in [0,\g], \ \Im s<0\}.
$$

To prove that the integrals on the arcs $D_R^-$ and $D_R^+$ vanish for $R\to \infty$ it is sufficient to show that for any $x>0$ the function $\widehat{w}(x,s)$ is uniformly bounded on $D_{R_n}^+$ and $D_{R_n}^-$, where $R_n\to\infty$, and that  $\widehat{w}(x,s)\to 0 $
for $s\in D_R^\pm$ and $R\to \infty$,
see e.g. \cite{Ditkin}, Chapter 2, Lemma 2.  
This follows from the fact that  
$\Re \sqrt{g(s)}\ge 0$ for $\Re s\ge 0$ 
and therefore 
\begin{equation}\label{estint}
\left|\frac{1}{s}\exp\left(-x\sqrt{g(s)}\right)\right|\le \frac{1}{R}\exp\left(-x\Re\sqrt{g(s)}\right)\le \frac{1}{R},\ \ s\in D_{R}^\pm.
\end{equation}
The integral on the semi-circular contour $D_\e$ equals $1/2$ when $\e\to 0$. This can be obtained by applying Jordan's lemma (or by direct check) and using that
$$
\lim_{s\to 0} s\left(\frac{1}{s}\exp\left(st-x\sqrt{g(s)}\right)\right)=1.
$$
 Integration on the contour $D$ yields after letting  $\e\to 0$ and $R\to\infty$:
$$
\frac{1}{2\pi\mathrm{i}}\int_D\frac{1}{s}\exp\left(st-x\sqrt{g(s)}\right)\,\mathrm{d}s=\frac1\pi \int_0^\infty \frac 1r\Im \exp\left(\mathrm{i}rt-x\sqrt{g(\mathrm{i}r)}\right) \,\mathrm{d}r.
$$
Here we have used the fact that $\sqrt{g(s^*)}=\left(\sqrt{g(s)}\right)^*$, where $*$ denotes the complex conjugate. 
Applying the formula for real and imaginary parts of the square root of a complex number we obtain the following result.

\begin{theorem}
The propagation function $w(x,t)$ admits the integral representation:
\begin{equation}\label{reshenie}
w(x,t)=\frac12+\frac 1\pi\int_0^\infty\exp(-x K^+(r))\sin(rt-xK^-(r))\,\frac{\mathrm{d}r}{r},\ \ x,t>0,
\end{equation}
where 
\begin{equation}\label{Kpm}
K^\pm(r)=\frac{1}{\sqrt{2}}\left(\left(A^2(r)+B^2(r)\right)^{1/2}\pm A(r)\right)^{1/2}
\end{equation}
with 
\begin{eqnarray}
&&A(r)=\Re g(ir)=cr^{\al}\cos(\al\pi/2)+\sum_{j=1}^m {c_j} r^{\al_j}\cos(\al_j\pi/2),\nonumber\\
&&B(r)=\Im g(ir)=cr^{\al}\sin(\al\pi/2)+\sum_{j=1}^m {c_j} r^{\al_j}\sin(\al_j\pi/2).\nonumber
\end{eqnarray}
\end{theorem}

To check that the obtained integral in (\ref{reshenie}) is convergent we note that $K^\pm(r)>0$, $K^\pm(r)\sim r^{\al_m/2}$ as $r\to 0$ and $K^\pm(r)\sim r^{\al/2}$ as $r\to \infty.$
Therefore, the function under the integral sign in (\ref{reshenie}) has an integrable singularity at $r=0$, while at $r\to\infty$ the term $\exp(-x K^+(r))$ ensures integrability  not only of this function, but also of its derivatives with respect to $t$. Therefore, $w(x,t)$ is well defined and infinitely differentiable in $t$ function.


\begin{cor}
In the single-term case $m=0$ and $c=1$ the propagation function $w(x,t)$ admits the integral representation:
\begin{equation}\label{reshenie1}
w(x,t)=\frac12+\frac 1\pi\int_0^\infty\exp\left(-x r^{\al/2}\cos(\al\pi/4)\right)\sin\left(rt-xr^{\al/2}\sin(\al\pi/4)\right)\,\frac{\mathrm{d}r}{r},\ \ x,t>0.
\end{equation}
\end{cor}

Based on (\ref{reshenie1}) and (\ref{CS}), representations for the fundamental solutions $\GG_c(x,t)$ and $\GG_s(x,t)$ can be easily derived after differentiation under the integral sign. The obtained in this way representations are different from those given in the works \cite{MainardiAML, Luchkopropspeed, LuchkoMainardi}. However, a numerical check shows that the different representations give identical results.

\section{Subordination to cosine families}

Assume $A$ is a closed linear unbounded operator densely defined in a Banach space $X$. Let $A$ generates a cosine family. This means that
the second-order Cauchy problem   
\begin{equation}\label{sec}
u_{tt}(\mathbf{x},t)=Au(\mathbf{x},t), \ \ t>0;\ \ \ u(\mathbf{x},0)=v(\mathbf{x})\in X,\  u_t(\mathbf{x},0)=0,
\end{equation}
is well posed. Denote by $S_2(t)$ the cosine family generated by the operator $A$, that is the solution operator for problem (\ref{sec}). The Laplace transform of the cosine family $S_2(t)$ generated by the operator $A$ is given by 
\begin{equation}\label{LapT}
\int_0^\infty e^{-st}S_2(t)\, \mbox{d}t=s(s^2-A)^{-1}.
\end{equation}
For details on cosine families we refer to \cite{Arendt}, Section 3.14.

We are concerned with the following problem for the multi-term time-fractional equation
\begin{equation}\label{A1}
cD_t^{\al}u(\mathbf{x},t)+\sum_{j=1}^m {c_j} D_t^{\al_j}u(\mathbf{x},t)=Au(\mathbf{x},t), \ \ t>0;\ \ \ u(\mathbf{x},0)=v(\mathbf{x})\in X,\  u_t(\mathbf{x},0)=0,
\end{equation}
where the parameters $\al,\al_j,c,c_j$ satisfy restrictions (\ref{al}).

It is convenient to rewrite problem (\ref{A1}) as an abstract Volterra  integral equation and use the definitions and some results from \cite{Pruss}. 
Applying  Laplace transform we obtain from (\ref{A1}) by the use of (\ref{LapD}) the integral equation
\begin{equation}\label{V}
u(\mathbf{x},t)=v(\mathbf{x})+\int_0^t k(t-\tau)Au(\mathbf{x},\tau)\,\mathrm{d}\tau, 
\end{equation}
where the scalar kernel $k(t)$ is defined by its Laplace transform 
\begin{equation}\label{kernel}
\widehat k(s)=1/g(s)
\end{equation} 
with function $g(s)$ defined in (\ref{gs}).

We will call problem (\ref{A1}) well posed if the corresponding Volterra integral equation (\ref{V}) is well posed.  In this case the resolvent for problem (\ref{V}) coincides with the solution operator of problem (\ref{A1}). Denote by $S(t)$ this solution operator.
Applying Laplace transform in (\ref{A1}) or (\ref{V}) it follows 
\begin{equation}\label{LapS}
\int_0^\infty e^{-st}S(t)\, \mbox{d}t=\frac{g(s)}{s}(g(s)-A)^{-1},
\end{equation}
where $g(s)$ is the function defined in (\ref{gs}).

Next the subordination relation (\ref{sub0}) will be proved. 
Let us define an operator-valued function $T(t)$ as follows
\begin{equation}\label{sub00}
T(t)=\int_0^\infty \ph(t,\tau)S_2(\tau)\, \mbox{d}\tau, \ \ t>0,
\end{equation}
where $S_2(t)$ is a cosine family generated by the operator $A$ and the function $\ph(t,\tau)$ is related to the propagation function $w(x,t)$ via the identity 
\begin{equation}\label{phi}
\ph(t,\tau)=- w_x(x,t)|_{x=\tau},\ \ t,\tau >0.
\end{equation}
The defined in this way function is p.d.f. in $\tau$. Indeed, 
Theorem \ref{33} implies that $\ph(t,\tau)\ge 0$. Moreover,
$$
\int_0^\infty \ph(t,\tau)\, \mbox{d}\tau=-\int_0^\infty w_x(x,t)\, \mbox{d}x=w(t,0)-w(t,\infty)=1
$$
where the boundary conditions (\ref{bc}) are taken into account.

 Application of the Laplace transform in (\ref{sub00}) gives by the use of (\ref{dw}) and (\ref{LapT})
\begin{eqnarray}
\int _0^\infty e^{-st}T(t)\, \mbox{d}t&=&\int_0^\infty\widehat{\ph}(s,\tau)S_2(\tau)\, \mbox{d}\tau\nonumber\\
&=&\frac{\sqrt{g(s)}}{s}\int_0^\infty \exp\left({-\tau \sqrt{g(s)}}\right)S_2(\tau)\, \mbox{d}\tau\nonumber\\
&=&\frac{g(s)}{s}(g(s)-A)^{-1}.\label{111}
\end{eqnarray}
Comparing (\ref{111}) to (\ref{LapS}), it follows by the uniqueness of the Laplace transform  that $T(t)=S(t)$. 

The fact that $\ph(t,\tau)$ is a p.d.f. has the following implication: if $S_2(t)$ is a bounded cosine family, such that $\|S_2(t)\|\le M, t\ge 0,$ then the same holds for $S(t).$ Indeed, from (\ref{sub00}) it follows
\begin{equation}\label{bounded}
\|S(t)\|\le  \int_0^\infty \ph(t,\tau)\|S_2(\tau)\|\, \mbox{d}\tau\le M \int_0^\infty \ph(t,\tau)\, \mbox{d}\tau=M,\ \ t\ge 0.
\end{equation}

Now we are ready to formulate our main result.

\begin{theorem}\label{thsub}
If $A$ is a generator of a bounded 
cosine family $S_2(t)$ in $X$ then
 problem (\ref{A1}) admits a bounded solution operator $S(t)$. It is related to $S_2(t)$ via the subordination identity 
\begin{equation}\label{sub000}
S(t)=\int_0^\infty \ph(t,\tau)S_2(\tau)\, \mbox{d}\tau, \ \ t>0.
\end{equation}
The function $\ph(t,\tau)$ is a p.d.f. in $\tau$ (i.e. conditions (\ref{PD}) hold) and admits the following integral representation 
\begin{eqnarray}
\ph(t,\tau)&=&\frac 1\pi\int_0^\infty\exp\left(-\tau K^+(r)\right)\left(K^+(r)\sin\left(rt-\tau K^-(r)\right)\right.\nonumber\\
&+&\left.K^-(r)\cos \left(rt-\tau K^-(r)\right)\right)\,\frac{\mbox{d}r}{r},\ \ t,\tau>0,\label{sf}
\end{eqnarray}
where $K^\pm(r)$ are the functions defined in (\ref{Kpm}).
\end{theorem}
\begin{proof}
The strict proof of the existence of solution operator $S(t)$ follows from Theorem~4.3 (iii) in \cite{Pruss}. The conditions of this theorem are satisfied since $\sqrt{g(s)}\in\BF$ and $\LL^{-1}\{1/\sqrt{g(s)}\}\sim t^{\al/2-1}$, as $t\to 0$, thus it is locally integrable. 

The integral representation (\ref{sf}) for the function (\ref{phi})  is obtained after easily justified differentiation under the integral sign in (\ref{reshenie}). 
\end{proof}

Plots of the p.d.f. $\ph(t,\tau)$ related to some two-term equations are shown in Figs.~4 and 5. The numerical computations are based on the integral representation (\ref{sf}).

In the case $\al=2$ identity (\ref{phi}) and Theorem \ref{w0} imply that $\ph(t,\tau)\equiv 0$ for $\tau>t/\sqrt{c}$. Therefore in this case the integral in (\ref{sub000}) is finite.

\begin{cor}
Let $\al=2$. Under the hypotheses of Theorem \ref{thsub} the subordination relation (\ref{sub000}) has the form
\begin{equation}\label{sub00f}
S(t)=\int_0^{t/\sqrt{c}} \ph(t,\tau)S_2(\tau)\, \mbox{d}\tau, \ \ t>0.
\end{equation}
 \end{cor}

For $\theta\in(0,\pi)$ let us denote by $\Si(\theta)$ the sector in the complex plane
$$\Si(\theta)=\{z\in\CC\backslash \{0\},\ |\arg z|<\theta\}.$$

Taking into account the asymptotic expansions of the functions $K^\pm(r)$, it is clear that the function under the integral sign in (\ref{sf}) can be infinitely differentiated in $t$. Therefore, this should hold also for the function $\ph(t,\tau)$. In the next theorem we prove a stronger regularity property in the case $\al<2$.

\begin{theorem}\label{an}
Assume $1<\al<2$ and let 
\begin{equation}\label{theta}
\theta_0=\frac{(2-\al)\pi}{2\al}-\e,
\end{equation}
where $\e>0$ is arbitrarily small. For any $\tau>0$ the function $\ph(t,\tau)$ as a function of $t$ admits analytic extension to the sector $\Si(\theta_0)$ and is bounded on each sector $\overline{ \Si(\theta)}$, $0<\theta<\theta_0$. 
\end{theorem}
\begin{proof}
First note that $\al>1$ implies $\theta_0<\pi/2$. It suffices to prove that for any $\tau>0$ the Laplace transform $\widehat{\ph}(s,\tau)$ of the function ${\ph}(t,\tau)$ admits analytic extension for $s\in \Si( \pi/2+\theta_0)$, such that $s\widehat{\ph}(s,\tau)$ is bounded for $s\in \overline{\Si( \pi/2+\theta)},\ 0<\theta<\theta_0$, (see e.g. \cite{Pruss}, Theorem 0.1).  

Indeed, since $\sqrt{g(s)}\in\CBF$, it can be extended analytically to $\CC\backslash (-\infty,0]$. Therefore this holds also for the function 
$$\widehat{\ph}(s,\tau)=\frac{\sqrt{g(s)}}{s}\exp(-\tau \sqrt{g(s)}).$$ 
For $s\in \overline{\Si( \pi/2+\theta)},\ \theta<\theta_0$, the definition (\ref{gs}) of $g(s)$ together with the property $|\arg(s_1+s_2)|\le \max \{|\arg s_1|,|\arg s_2|\}$ and (\ref{theta}) implies
$$
|\arg\sqrt{g(s)}|\le\frac{\al}{2} |\arg s|<\pi/2-\e\al/2.
$$
Therefore, 
$$
\left|s\widehat{\ph}(s,\tau)\right|=\left|\sqrt{g(s)}\exp\left(-\tau \sqrt{g(s)}\right)\right|\le \rho \exp\left(-\tau \rho \cos\left(\arg\sqrt{g(s)}\right)\right)\le \rho e^{-a \rho }\le (ea)^{-1},
$$
where $\rho=\left|\sqrt{g(s)}\right|$ and $a=\tau \sin(\e\al/2)>0$.
\end{proof}

\begin{theorem}\label{cor}
 Let $1<\al<2$. Under the hypotheses of Theorem \ref{thsub} the solution operator $S(t)$ of problem (\ref{A1}) admits analytic extension to the sector $\Si(\theta_0)$, where $\theta_0$ is defined in 
(\ref{theta}).
\end{theorem}
\begin{proof}
Since $S_2(t)$ is bounded, according to Theorem \ref{an} the function under the integral sign in (\ref{sub000}) is analytic in $t\in \Si(\theta_0)$ and the integral is absolutely and uniformly convergent on compact subsets of $\Si(\theta_0)$. Therefore, $S(t)$ given by (\ref{sub000}) is analytic in $\Si(\theta_0)$.
\end{proof}

Theorem \ref{cor} is in agreement with Theorem 3.3. in \cite{B}, where the same property is established for the solution operators $S_\al(t)$ of problem (\ref{st}).




\section{Subordination to $S_\al(t)$}
 
In fact, the solution operator $S(t)$ of problem (\ref{A1}) is not only subordinate to the cosine families $S_2(t)$, but also to the solution operator $S_\al(t)$ of problem (\ref{st}), which is a stronger result (see the remark below). The proof follows the same steps as in the case $\al=2$ above. We only need to prove the following property of $g(s)$. 

\begin{prop}\label{alal}
If $g(s)$ is defined as in (\ref{gs}) with parameters $\al,\al_j,c,c_j,j=1,...,m,$ satisfying (\ref{al}), then $g(s)^{1/\al}\in \CBF$.
\end{prop}
\begin{proof}
For the proof we adapt a method proposed in \cite{GLS}.
It is sufficient to show that 
\begin{equation}\label{sssf}
f(s)=\frac{g(s)^{1/\al}}{s}\in\SF.
\end{equation}
Since $0<\al-\al_m<1$ the function $$f^\al(s)=\frac{g(s)}{s^\al}=c+\sum_{j=1}^m {c_j} s^{\al_j-\al}$$  is a Stieltjes function. Moreover, ${s}^{1/\al}\in\CBF$  for $\al>1$. This together with the first composition property in (D) from the Appendix gives (\ref{sssf}). 
\end{proof}
Let us note that the property $g(s)^{1/\al}\in \CBF$ is stronger than the property $g(s)^{1/2}\in \CBF$ proven in Proposition~\ref{pg}. This follows from the representation $g(s)^{1/2}=(g(s)^{1/\al})^{\al/2}$ as a composition of two complete Bernstein functions, which by the second property in (D) from the Appendix is again a complete Bernstein function.

\begin{theorem}\label{thsubal}
Assume problem (\ref{st}) has a bounded 
solution operator $S_\al(t)$. Then
 problem (\ref{A1}) admits a bounded solution operator $S(t)$, which is related to $S_\al(t)$ by the subordination identity 
\begin{equation}\label{subal}
S(t)=\int_0^\infty \psi(t,\tau)S_\al(\tau)\, \mbox{d}\tau, \ \ t>0,
\end{equation}
 where the function $\psi(t,\tau)$ is a p.d.f. in $\tau$.
\end{theorem}
\begin{proof}
The strict proof follows from  Proposition \ref{alal} and Theorem~4.3 (iii) in \cite{Pruss}. Here we will give only the main steps.
The function $\psi(t,\tau)$ is defined as the inverse Laplace transform 
\begin{equation}\label{defpsiint}
\psi(t,\tau)=\frac{1}{2\pi \mathrm{i}} \int_{\g-\mathrm{i}\infty}^{\g+\mathrm{i}\infty}\frac{g(s)^{1/\al}}{s}\exp\left(st-\tau g(s)^{1/\al}\right)\,\mbox{d}s, \ \ \g,t,\tau>0.
\end{equation}
By Proposition \ref{alal} and Bernstein theorem $\psi(t,\tau)\ge 0$. Moreover,
\begin{eqnarray}
\int_0^\infty \psi(t,\tau)\mbox{d}\tau&=&\frac{1}{2\pi \mathrm{i}} \int_{\g-\mathrm{i}\infty}^{\g+\mathrm{i}\infty}e^{st}\frac{g(s)^{1/\al}}{s}\int_0^\infty\exp\left(-\tau g(s)^{1/\al}\right)\,\mbox{d}\tau\mbox{d}s\nonumber\\
&=&\frac{1}{2\pi \mathrm{i}} \int_{\g-\mathrm{i}\infty}^{\g+\mathrm{i}\infty}\frac{e^{st}}{s}\mbox{d}s=1.\nonumber
\end{eqnarray} 
 Therefore, $\psi(t,\tau)$ is a p.d.f. in $\tau$ and the boundedness of $S(t)$ follows easily from the boundedness of $S_\al(t)$ as in the previous section.

To prove that the operator $S(t)$ defined in (\ref{subal}) is the solution operator of problem (\ref{A1}), we need to check (\ref{LapS}). Indeed, since (see e.g. \cite{B})
\begin{equation}\label{LapSal}
\int_0^\infty e^{-st}S_\al(t)\, \mbox{d}t=s^{\al-1}(s^\al-A)^{-1},
\end{equation}
then (\ref{defpsiint}) and (\ref{LapSal}) imply
\begin{eqnarray}
\int _0^\infty e^{-st}S(t)\, \mbox{d}t&=&\int_0^\infty\widehat{\psi}(s,\tau)S_\al(\tau)\, \mbox{d}\tau\nonumber\\
&=&\frac{g(s)^{1/\al}}{s}\int_0^\infty \exp\left({-\tau g(s)^{1/\al}}\right)S_\al(\tau)\, \mbox{d}\tau\nonumber\\
&=&\frac{g(s)}{s}(g(s)-A)^{-1}.\nonumber
\end{eqnarray}
\end{proof}

Theorem~\ref{thsubal} implies that the solution operator $S(t)$ has (at least) the same regularity as $S_\al(t)$. This result is in agreement with Theorem~3.4 in \cite{KosticFCAA}.

 \section{Examples}

Two simple examples of application of Theorem~\ref{thsub} are given in this section. Let us note that the function $\phi(t,\tau)$ in the subordination identity (\ref{sub000}) does not depend on the operator $A$.

First, let $X=L^p(\RR), 1\le p<\infty$. Define the operator $A$ by means of $(Au)(x)=u''(x)$, with domain $D(A)=\{u\in X: u', u''\in X, \ u(\pm\infty)=0 \}$. Then $A$ generates a bounded cosine family given by the d'Alembert formula
\begin{equation}\label{Dal}
(S_2(t)v)(x)=\frac12\left(v(x+t)+v(x-t)\right).
\end{equation}
Inserting (\ref{Dal}) in the subordination formula (\ref{sub000}) we obtain for the solution of problem (\ref{A1}) 
\begin{equation}\label{Ex1}
u(x,t)=(S(t)v)(x)=\int_0^\infty \ph(t,\tau)(S_2(\tau)v)(x)\, \mbox{d}\tau= \frac12\int_{-\infty}^\infty \ph(t,|\xi|)v(x-\xi)\, \mbox{d}\xi.
\end{equation}
In this way the representation (\ref{Gphi}) of the fundamental solution of the spatially one-dimensional Cauchy problem is 
established. It is remarkable that, due to the specific form of the d'Alembert formula (\ref{Dal}), the convolution in time in subordination relation (\ref{sub000}) is transformed to a convolution relation for the space variable in (\ref{Ex1}).

For the second example, assume $\Om\subset\RR^n$ is an open set and let $X=L^2(\Om)$. Let $A$ be the Laplace operator with Dirichlet boundary conditions: $A=\De_{\mathbf{x}}$, $D(A)= H_0^1(\Omega) \cap H^2(\Omega)$. It is known that the operator $A$ generates a bounded cosine family, see e.g \cite{Arendt}, Section 7.2.

If $\{-\la_n,\phh_n\}_{n=1}^\infty$ is the eigensystem of the operator $A$, then $0<\la_1\le\la_2\le...,\ \la_n\to\infty$ as $n\to\infty$, and $ \{\phh_n\}_{n=1}^\infty$ form an orthonormal basis of $L^2(\Om)$.
The cosine family $S_2(t)$ admits the following eigenfunction decomposition
\begin{equation}\label{opTTT}
S_2(t)v=\sum_{n=1}^\infty v_n \cos(\sqrt{\la_n} t)\phh_n,
\end{equation} 
with 
$
v_n=(v,\phh_n),
$ 
where $(.,.)$ is the inner product in $L^2(\Om)$

Therefore, applying Theorem \ref{thsub} we obtain
the solution of problem (\ref{A1}) in the form:
\begin{equation}\label{FS}
S(t)v=\sum_{n=1}^\infty v_n u_n(t)\phh_n,
\end{equation}
where the eigenmodes $u_n(t)$ admit the integral representation
 \begin{equation}\label{Sncos}
u_n(t)=\int_0^\infty \phi(t,\tau)\cos(\sqrt{\la_n} \tau)\, \mbox{d}\tau.
\end{equation} 
The eigenmodes $u_n(t)$ can be numerically computed by the use of (\ref{Sncos}) and (\ref{sf}). 

In particular, in the one-dimensional case, $\Om=(0,1)$, the eigensystem is $\la_n=n^2\pi^2$,  $\phh_n=\sqrt{2}\sin(n\pi x)$, $n=1,2,...$.  In Fig.~6 we present plots of the first four eigenmodes $u_n(t)$ for the two-term one-dimensional equation, which are computed using formula (\ref{Sncos}).

\section{Conclusions}

 Subordination principle is derived for the fractional evolution equation with a discrete distribution of Caputo time-derivatives such that the largest and the smallest orders, $\al$ and $\al_m$,  satisfy the conditions $1<\al\le 2$ and $\al-\al_m\le 1$. The subordination identity splits the solution into two parts. The first part (the p.d.f.) depends only on the parameters of the distribution of time-fractional derivatives and the second part is the cosine family generated by the spatial operator $A$. The probability density function is closely related to the fundamental solution of the corresponding one-dimensional Cauchy problem. An explicit representation of this function is given and its regularity is studied and applied to obtain regularity results for the solution of the general problem.

An interesting phenomenon is established in the case $\al=2$ and at least one more time-derivative of noninteger order: coexistence of finite wave speed and absence of wave front. This is a memory effect, not observed in linear integer-order differential equations.

The proofs in this work are essentially based on the fact that the function $\sqrt{g(s)}$ is a Bernstein function, which is ensured by the assumption $\al-\al_m\le 1$. Whether and to what extent this condition can be relaxed 
is, to the best knowledge of the authors, an open problem. 

The obtained results can be generalized to the case when the orders of fractional derivatives are continuously distributed over an interval $(a,b)\subset(0,2]$ with $b-a\le 1$.

\section*{Appendix}

Here we list definitions and some properties of special classes of functions related to Bernstein functions. 

A function $\phh:(0,\infty)\to \RR$ is said to be completely monotone function ($\phh\in \CMF$)  if it is of class $C^\infty$ and
\begin{equation}\label{CMF}
(-1)^n \phh^{(n)}(\la)\ge 0,\ \ \la>0, n=0,1,2,...
\end{equation}
The function $\la^{\al-1}, \al\in[0,1]$, and the exponential function $\exp(-a\la), \ a>0,$ 
are basic examples of completely monotone functions.

The characterization of the class $\CMF$ is given by the Bernstein's theorem (see e.g. \cite{Feller}) which states that 
a function is completely monotone if and only if it can be represented as the Laplace transform of a non-negative measure (non-negative function or generalized function).

The class of Stieltjes functions ($\SF$) consists of all functions defined on $(0,\infty)$ which can be written as a restriction of the Laplace transform of a completely monotone function to the real positive semi-axis.
Obviously, $\SF \subset\CMF$. The function $\la^{\al-1}, \al\in[0,1]$, is a basic example of Stieltjes function.

A non-negative function $\phh$ on $(0,\infty)$ is said to be a Bernstein function ($\phh\in \BF$) if 
$\phh'(\la)\in \CMF$; $\phh(\la)$ is said to be a complete Bernstein functions ($\CBF$) if and only if $\phh(\la)/\la\in \SF$. 
We have the inclusion $\CBF\subset \BF.$ The function $\la^{\al}, \al\in[0,1]$, is a basic example of a complete Bernstein function.

 A selection of properties is listed next:\\
{\rm(A)} 
The class $\CMF$ is closed under point-wise addition and multiplication.\\
{\rm(B)} If $\phh\in \BF$ then $\phh(\la)/\la\in \CMF$.\\
{\rm(C)} If $\phh\in \CMF$ and $\psi\in \BF$ then the composite function $\phh(\psi)\in\CMF$.\\
{\rm(D)} Let $\phh\in \CBF$. If $\psi\in \SF$ then $\phh(\psi)\in\SF$; if $\psi\in \CBF$ then $\phh(\psi)\in\CBF$.\\
{\rm(E)} For not identically vanishing functions $\phh$ and $1/\phh$: $\phh\in \CBF$ if and only if $1/\phh\in\SF$.
\\
{\rm(F)} Let $p,q\in(0,1)$ and $p+q\le 1$. Then $\phh_1^p.\phh_2^q\in \CBF$ for all $\phh_1,\phh_2\in \CBF$. 
\\
{\rm(G)} If $\phh\in\SF$ or $\phh\in\CBF$ then it can be analytically extended to $\CC\backslash (-\infty,0]$ and $$|\arg \phh(z)|\le |\arg z|,\ z\in \CC\backslash (-\infty,0].$$
For proofs and more details on these special classes of functions we refer to \cite{CMF}, see also  \cite{GLS}.

\section*{Acknowledgments}

This work is partially supported by Bulgarian National Science Fund (Grant DFNI-I02/9); and performed in the frames of the bilateral research project between Bulgarian and Serbian academies of sciences "Analytical and numerical methods for differential and integral equations and mathematical models of arbitrary (fractional or high integer) order".


\bibliographystyle{model1-num-names}
\bibliography{paperJCAM}

\begin{thebibliography}{43}
\expandafter\ifx\csname natexlab\endcsname\relax\def\natexlab#1{#1}\fi
\providecommand{\bibinfo}[2]{#2}
\ifx\xfnm\relax \def\xfnm[#1]{\unskip,\space#1}\fi
\bibitem[{Mainardi(1996)}]{MainardiAML}
\bibinfo{author}{F.~Mainardi},
\newblock \bibinfo{title}{The fundamental solutions for the fractional
  diffusion-wave equation},
\newblock \bibinfo{journal}{Appl. Math. Lett.} \bibinfo{volume}{9}
  (\bibinfo{year}{1996}) \bibinfo{pages}{23--28}.
\bibitem[{Mainardi and Paradisi(2001)}]{MainardiFDW}
\bibinfo{author}{F.~Mainardi}, \bibinfo{author}{P.~Paradisi},
\newblock \bibinfo{title}{Fractional diffusive waves},
\newblock \bibinfo{journal}{Journal of Computational Acoustics,}
  \bibinfo{volume}{9} (\bibinfo{year}{2001}) \bibinfo{pages}{1417--1436}.
\bibitem[{Mainardi(2010)}]{M}
\bibinfo{author}{F.~Mainardi}, \bibinfo{title}{Fractional Calculus and Waves in
  Linear Viscoelasticity}, \bibinfo{publisher}{Imperial College Press},
  \bibinfo{address}{London}, \bibinfo{year}{2010}.
\bibitem[{Atanackovi\'{c} et~al.(2014)Atanackovi\'{c}, Pilipovi\'{c},
  Stankovi\'{c}, and Zorica}]{A}
\bibinfo{author}{T.~Atanackovi\'{c}}, \bibinfo{author}{S.~Pilipovi\'{c}},
  \bibinfo{author}{B.~Stankovi\'{c}}, \bibinfo{author}{D.~Zorica},
  \bibinfo{title}{Fractional Calculus with Applications in Mechanics:
  Vibrations and Diffusion Processes}, \bibinfo{publisher}{John Wiley \& Sons},
  \bibinfo{address}{London}, \bibinfo{year}{2014}.
\bibitem[{Luchko et~al.(2013)Luchko, Mainardi, and Povstenko}]{Luchkopropspeed}
\bibinfo{author}{Y.~Luchko}, \bibinfo{author}{F.~Mainardi},
  \bibinfo{author}{Y.~Povstenko},
\newblock \bibinfo{title}{Propagation speed of the maximum of the fundamental
  solution to the fractional diffusion-wave equation},
\newblock \bibinfo{journal}{Comput. Math. Appl.} \bibinfo{volume}{66}
  (\bibinfo{year}{2013}) \bibinfo{pages}{774--784}.
\bibitem[{Luchko and Mainardi(2014)}]{LuchkoMainardi}
\bibinfo{author}{Y.~Luchko}, \bibinfo{author}{F.~Mainardi},
\newblock \bibinfo{title}{Propagation speed of the maximum of the fundamental
  solution to the fractional diffusion-wave equation},
\newblock \bibinfo{journal}{ASME Journal of Vibration and Acoustics}
  \bibinfo{volume}{136} (\bibinfo{year}{2014}) \bibinfo{pages}{051008}.
\bibitem[{Meerschaert et~al.(2015)Meerschaert, Schilling, and Sikorskii}]{sub2}
\bibinfo{author}{M.~Meerschaert}, \bibinfo{author}{R.~Schilling},
  \bibinfo{author}{A.~Sikorskii},
\newblock \bibinfo{title}{Stochastic solutions for fractional wave equations},
\newblock \bibinfo{journal}{Nonlinear Dynam.} \bibinfo{volume}{80}
  (\bibinfo{year}{2015}) \bibinfo{pages}{1685--1695}.
\bibitem[{Duck(1990)}]{Duck}
\bibinfo{author}{F.~Duck}, \bibinfo{title}{Physical Properties of Tissue: A
  Comprehensive Reference Book}, \bibinfo{publisher}{Academic Press},
  \bibinfo{address}{Boston}, \bibinfo{year}{1990}.
\bibitem[{Atanackovi\'c et~al.(2009{\natexlab{a}})Atanackovi\'c, Pilipovi\'c,
  and Zorica}]{A1}
\bibinfo{author}{T.~Atanackovi\'c}, \bibinfo{author}{S.~Pilipovi\'c},
  \bibinfo{author}{D.~Zorica},
\newblock \bibinfo{title}{Time distributed-order diffusion-wave equation. {I}.
  {V}olterra-type equation},
\newblock \bibinfo{journal}{Proc. R. Soc. A} \bibinfo{volume}{465}
  (\bibinfo{year}{2009}{\natexlab{a}}) \bibinfo{pages}{1869--1891}.
\bibitem[{Atanackovi\'c et~al.(2009{\natexlab{b}})Atanackovi\'c, Pilipovi\'c,
  and Zorica}]{A2}
\bibinfo{author}{T.~Atanackovi\'c}, \bibinfo{author}{S.~Pilipovi\'c},
  \bibinfo{author}{D.~Zorica},
\newblock \bibinfo{title}{Time distributed-order diffusion-wave equation. {
  II}. {A}pplications of {L}aplace and {F}ourier transformations},
\newblock \bibinfo{journal}{Proc. R. Soc. A} \bibinfo{volume}{465}
  (\bibinfo{year}{2009}{\natexlab{b}}) \bibinfo{pages}{1893--1917}.
\bibitem[{Sandev et~al.(2017)Sandev, Sokolov, Metzler, and Chechkin}]{Trifce}
\bibinfo{author}{T.~Sandev}, \bibinfo{author}{I.~M. Sokolov},
  \bibinfo{author}{R.~Metzler}, \bibinfo{author}{A.~Chechkin},
\newblock \bibinfo{title}{Beyond monofractional kinetics},
\newblock \bibinfo{journal}{Chaos, Solitons \& Fractals} \bibinfo{volume}{102}
  (\bibinfo{year}{2017}) \bibinfo{pages}{210--217}.
\bibitem[{Chen and Holm(2004)}]{FractLap}
\bibinfo{author}{W.~Chen}, \bibinfo{author}{S.~Holm},
\newblock \bibinfo{title}{Fractional {L}aplacian time-space models for linear
  and nonlinear lossy media exhibiting arbitrary frequency power-law
  dependency},
\newblock \bibinfo{journal}{J. Acoust. Soc. Am.} \bibinfo{volume}{115}
  (\bibinfo{year}{2004}) \bibinfo{pages}{1424--1430}.
\bibitem[{Treeby and Cox(2010)}]{FractLap1}
\bibinfo{author}{B.~E. Treeby}, \bibinfo{author}{B.~T. Cox},
\newblock \bibinfo{title}{Modeling power law absorption and dispersion for
  acoustic propagation using the fractional {L}aplacian},
\newblock \bibinfo{journal}{J. Acoust. Soc. Am.} \bibinfo{volume}{127}
  (\bibinfo{year}{2010}) \bibinfo{pages}{2741--2748}.
\bibitem[{Sakamoto and Yamamoto(2012)}]{SaYa}
\bibinfo{author}{K.~Sakamoto}, \bibinfo{author}{M.~Yamamoto},
\newblock \bibinfo{title}{Initial value/boundary value problems for fractional
  diffusionwave equations and applications to some inverse problems},
\newblock \bibinfo{journal}{J. Math. Anal. Appl.} \bibinfo{volume}{382}
  (\bibinfo{year}{2012}) \bibinfo{pages}{426--447}.
\bibitem[{Bajlekova(2001)}]{B}
\bibinfo{author}{E.~Bajlekova}, \bibinfo{title}{Fractional Evolution Equations
  in Banach Spaces}, Ph.D. thesis, \bibinfo{address}{Eindhoven, The
  Netherlands}, \bibinfo{year}{2001}.
\bibitem[{Li et~al.(2012)Li, Kosti\'c, Li, and Piskarev}]{KosticFCAA}
\bibinfo{author}{C.-G. Li}, \bibinfo{author}{M.~Kosti\'c},
  \bibinfo{author}{M.~Li}, \bibinfo{author}{S.~Piskarev},
\newblock \bibinfo{title}{On a class of time-fractional differential
  equations},
\newblock \bibinfo{journal}{Fract. Calc. Appl. Anal.} \bibinfo{volume}{15}
  (\bibinfo{year}{2012}) \bibinfo{pages}{639--668}.
\bibitem[{Arendt et~al.(2011)Arendt, Batty, Hieber, and Neubrander}]{Arendt}
\bibinfo{author}{W.~Arendt}, \bibinfo{author}{C.~Batty},
  \bibinfo{author}{M.~Hieber}, \bibinfo{author}{F.~Neubrander},
  \bibinfo{title}{Vector-valued {L}aplace transforms and {C}auchy problems},
  \bibinfo{publisher}{Birkh{\"a}user}, \bibinfo{address}{Basel},
  \bibinfo{year}{2011}.
\bibitem[{Bazhlekova(2000)}]{Subordination0}
\bibinfo{author}{E.~Bazhlekova},
\newblock \bibinfo{title}{Subordination principle for fractional evolution
  equations},
\newblock \bibinfo{journal}{Fract. Calc. Appl. Anal.} \bibinfo{volume}{3}
  (\bibinfo{year}{2000}) \bibinfo{pages}{213--230}.
\bibitem[{Pr{\"u}ss(1993)}]{Pruss}
\bibinfo{author}{J.~Pr{\"u}ss}, \bibinfo{title}{Evolutionary {I}ntegral
  {E}quations and {A}pplications}, \bibinfo{publisher}{Birkh{\"a}user},
  \bibinfo{address}{Basel}, \bibinfo{year}{1993}.
\bibitem[{Kochubei(2014)}]{sub1}
\bibinfo{author}{A.~Kochubei},
\newblock \bibinfo{title}{Asymptotic properties of solutions of the fractional
  diffusion-wave equation},
\newblock \bibinfo{journal}{Fract. Calc. Appl. Anal.} \bibinfo{volume}{17}
  (\bibinfo{year}{2014}) \bibinfo{pages}{881--896}.
\bibitem[{Keyantuo et~al.(2016)Keyantuo, Lizama, and Warma}]{KeLiWa}
\bibinfo{author}{V.~Keyantuo}, \bibinfo{author}{C.~Lizama},
  \bibinfo{author}{M.~Warma},
\newblock \bibinfo{title}{Existence, regularity and representation of solutions
  of time fractional diffusion equations},
\newblock \bibinfo{journal}{Adv. Differential Equations} \bibinfo{volume}{21}
  (\bibinfo{year}{2016}) \bibinfo{pages}{837--886}.
\bibitem[{Miller and Yamamoto(2013)}]{sub0}
\bibinfo{author}{L.~Miller}, \bibinfo{author}{M.~Yamamoto},
\newblock \bibinfo{title}{Coefficient inverse problem for a fractional
  diffusion equation},
\newblock \bibinfo{journal}{Inverse Probl.} \bibinfo{volume}{29}
  (\bibinfo{year}{2013}) \bibinfo{pages}{075013}.
\bibitem[{Alvarez-Pardo and Lizama(2014)}]{APLizama}
\bibinfo{author}{E.~Alvarez-Pardo}, \bibinfo{author}{C.~Lizama},
\newblock \bibinfo{title}{Mild solutions for multi-term time-fractional
  differential equations with nonlocal initial conditions},
\newblock \bibinfo{journal}{Electron. J. Diff. Eqns.} \bibinfo{volume}{39}
  (\bibinfo{year}{2014}) \bibinfo{pages}{1--10}.
\bibitem[{Kochubei(2008)}]{K}
\bibinfo{author}{A.~Kochubei},
\newblock \bibinfo{title}{Distributed order calculus and equations of ultraslow
  diffusion},
\newblock \bibinfo{journal}{J. Math. Anal. Appl.} \bibinfo{volume}{340}
  (\bibinfo{year}{2008}) \bibinfo{pages}{252--281}.
\bibitem[{Meerschaert et~al.(2011)Meerschaert, Nane, and Vellaisamy}]{Meer}
\bibinfo{author}{M.~Meerschaert}, \bibinfo{author}{E.~Nane},
  \bibinfo{author}{P.~Vellaisamy},
\newblock \bibinfo{title}{Distributed-order fractional diffusions on bounded
  domains},
\newblock \bibinfo{journal}{J. Math. Anal. Appl.} \bibinfo{volume}{379}
  (\bibinfo{year}{2011}) \bibinfo{pages}{216--228}.
\bibitem[{Bazhlekova(2015)}]{ITSF}
\bibinfo{author}{E.~Bazhlekova},
\newblock \bibinfo{title}{Completely monotone functions and some classes of
  fractional evolution equations},
\newblock \bibinfo{journal}{Integr. Transf. Spec. F.} \bibinfo{volume}{26}
  (\bibinfo{year}{2015}) \bibinfo{pages}{737--752}.
\bibitem[{Mijena and Nane(2014)}]{Mijena}
\bibinfo{author}{J.~Mijena}, \bibinfo{author}{E.~Nane},
\newblock \bibinfo{title}{Strong analytic solutions of fractional {C}auchy
  problems},
\newblock \bibinfo{journal}{Proceedings of the American Mathematical Society}
  \bibinfo{volume}{142} (\bibinfo{year}{2014}) \bibinfo{pages}{1717--1731}.
\bibitem[{Abadias and Miana(2015)}]{Abadias}
\bibinfo{author}{L.~Abadias}, \bibinfo{author}{P.~J. Miana},
\newblock \bibinfo{title}{A subordination principle on {W}right functions and
  regularized resolvent families},
\newblock \bibinfo{journal}{Journal of Function Spaces}  (\bibinfo{year}{2015})
  \bibinfo{pages}{158145}.
\bibitem[{Feller(1971)}]{Feller}
\bibinfo{author}{W.~Feller}, \bibinfo{title}{An introduction to probability
  theory and its applications}, volume~\bibinfo{volume}{2},
  \bibinfo{publisher}{Wiley}, \bibinfo{address}{New York},
  \bibinfo{year}{1971}.
\bibitem[{Atanackovi\'c et~al.(2007)Atanackovi\'c, Pilipovi\'c, and
  Zorica}]{At2term}
\bibinfo{author}{T.~Atanackovi\'c}, \bibinfo{author}{S.~Pilipovi\'c},
  \bibinfo{author}{D.~Zorica},
\newblock \bibinfo{title}{Diffusion wave equation with two fractional
  derivatives of different order},
\newblock \bibinfo{journal}{Journal of Physics A: Mathematical and Theoretical}
  \bibinfo{volume}{40} (\bibinfo{year}{2007}) \bibinfo{pages}{5319--5333}.
\bibitem[{Orsingher and Beghin(2004)}]{OB}
\bibinfo{author}{E.~Orsingher}, \bibinfo{author}{L.~Beghin},
\newblock \bibinfo{title}{Time-fractional telegraph equations and telegraph
  processes with brownian time},
\newblock \bibinfo{journal}{Probab. Theory Relat. Fields} \bibinfo{volume}{128}
  (\bibinfo{year}{2004}) \bibinfo{pages}{141--160}.
\bibitem[{Mamchuev(2017)}]{Mamchuev}
\bibinfo{author}{M.~O. Mamchuev},
\newblock \bibinfo{title}{Solutions of the main boundary value problems for the
  time-fractional telegraph equation by the green function method},
\newblock \bibinfo{journal}{Fract. Calc. Appl. Anal.} \bibinfo{volume}{20}
  (\bibinfo{year}{2017}) \bibinfo{pages}{190--211}.
\bibitem[{Qi and Guo(2014)}]{Cattaneo}
\bibinfo{author}{H.~Qi}, \bibinfo{author}{X.~Guo},
\newblock \bibinfo{title}{Transient fractional heat conduction with generalized
  cattaneo model},
\newblock \bibinfo{journal}{International Journal of Heat and Mass Transfer}
  \bibinfo{volume}{76} (\bibinfo{year}{2014}) \bibinfo{pages}{535--539}.
\bibitem[{J.Chen et~al.(2012)J.Chen, Li, Anh, Shen, Liu, and Liao}]{damping}
\bibinfo{author}{J.Chen}, \bibinfo{author}{F.~Li}, \bibinfo{author}{V.~Anh},
  \bibinfo{author}{S.~Shen}, \bibinfo{author}{Q.~Liu},
  \bibinfo{author}{C.~Liao},
\newblock \bibinfo{title}{The analytical solution and numerical solution of the
  fractional diffusion-wave equation with damping},
\newblock \bibinfo{journal}{Appl. Math. Comput.} \bibinfo{volume}{219}
  (\bibinfo{year}{2012}) \bibinfo{pages}{1737--1748}.
\bibitem[{Qi et~al.(2013)Qi, Xu, and Guo}]{Qi}
\bibinfo{author}{H.-T. Qi}, \bibinfo{author}{H.-Y. Xu}, \bibinfo{author}{X.-W.
  Guo},
\newblock \bibinfo{title}{The {C}attaneo-type time fractional heat conduction
  equation for laser heating},
\newblock \bibinfo{journal}{Comput. Math. Appl.} \bibinfo{volume}{66}
  (\bibinfo{year}{2013}) \bibinfo{pages}{824--831}.
\bibitem[{Jiang et~al.(2012)Jiang, Liu, Turner, and Burrage}]{JLTB}
\bibinfo{author}{H.~Jiang}, \bibinfo{author}{F.~Liu},
  \bibinfo{author}{I.~Turner}, \bibinfo{author}{K.~Burrage},
\newblock \bibinfo{title}{Analytical solutions for the multi-term
  time-fractional diffusion-wave/diffusion equations in a finite domain},
\newblock \bibinfo{journal}{Comput. Math. Appl.} \bibinfo{volume}{64}
  (\bibinfo{year}{2012}) \bibinfo{pages}{3377–--3388}.
\bibitem[{Dehghan et~al.(2015)Dehghan, Safarpoor, and Abbaszadeh}]{DehghanJCAM}
\bibinfo{author}{M.~Dehghan}, \bibinfo{author}{M.~Safarpoor},
  \bibinfo{author}{M.~Abbaszadeh},
\newblock \bibinfo{title}{Two high-order numerical algorithms for solving the
  multi-term time fractional diffusion-wave equations},
\newblock \bibinfo{journal}{J. Comput. Appl. Math.} \bibinfo{volume}{290}
  (\bibinfo{year}{2015}) \bibinfo{pages}{174--195}.
\bibitem[{Lizama(2012)}]{L13}
\bibinfo{author}{C.~Lizama},
\newblock \bibinfo{title}{Solutions of two-term time fractional order
  differential equations with nonlocal initial conditions},
\newblock \bibinfo{journal}{Electron. J. Qual. Theory Differ. Equat.}
  \bibinfo{volume}{82} (\bibinfo{year}{2012}) \bibinfo{pages}{1--9}.
\bibitem[{Keyantuo et~al.(2013)Keyantuo, Lizama, and Warma}]{KLW10}
\bibinfo{author}{V.~Keyantuo}, \bibinfo{author}{C.~Lizama},
  \bibinfo{author}{M.~Warma},
\newblock \bibinfo{title}{Asymptotic behavior of fractional order semilinear
  evolution equations},
\newblock \bibinfo{journal}{Differential and Integral Equations}
  \bibinfo{volume}{26} (\bibinfo{year}{2013}) \bibinfo{pages}{757--780}.
\bibitem[{Schilling et~al.(2010)Schilling, Song, and Vondra\v{c}ek}]{CMF}
\bibinfo{author}{R.~Schilling}, \bibinfo{author}{R.~Song},
  \bibinfo{author}{Z.~Vondra\v{c}ek}, \bibinfo{title}{Bernstein functions:
  Theory and applications}, \bibinfo{publisher}{De Gruyter},
  \bibinfo{address}{Berlin}, \bibinfo{year}{2010}.
\bibitem[{Gorenflo et~al.(2013)Gorenflo, Luchko, and Stojanovi\'c}]{GLS}
\bibinfo{author}{R.~Gorenflo}, \bibinfo{author}{Y.~Luchko},
  \bibinfo{author}{M.~Stojanovi\'c},
\newblock \bibinfo{title}{Fundamental solution of a distributed order
  time-fractional diffusion-wave equation as probability density},
\newblock \bibinfo{journal}{Fract. Calc. Appl. Anal.} \bibinfo{volume}{16}
  (\bibinfo{year}{2013}) \bibinfo{pages}{297--316}.
\bibitem[{Bazhlekova and Bazhlekov(2017)}]{CAMWA}
\bibinfo{author}{E.~Bazhlekova}, \bibinfo{author}{I.~Bazhlekov},
\newblock \bibinfo{title}{Unidirectional flows of fractional {J}effreys’
  fluids: {T}hermodynamic constraints and subordination},
\newblock \bibinfo{journal}{Comput. Math. Appl.} \bibinfo{volume}{73}
  (\bibinfo{year}{2017}) \bibinfo{pages}{1363--1376}.
\bibitem[{Ditkin and Prudnikov(1965)}]{Ditkin}
\bibinfo{author}{V.~A. Ditkin}, \bibinfo{author}{A.~P. Prudnikov},
  \bibinfo{title}{Integral transforms and operational calculus},
  \bibinfo{publisher}{Pergamon Press}, \bibinfo{address}{Oxford, New York},
  \bibinfo{year}{1965}.

\end{thebibliography}


\begin{figure}[h]
		\includegraphics[width=0.8\textwidth]{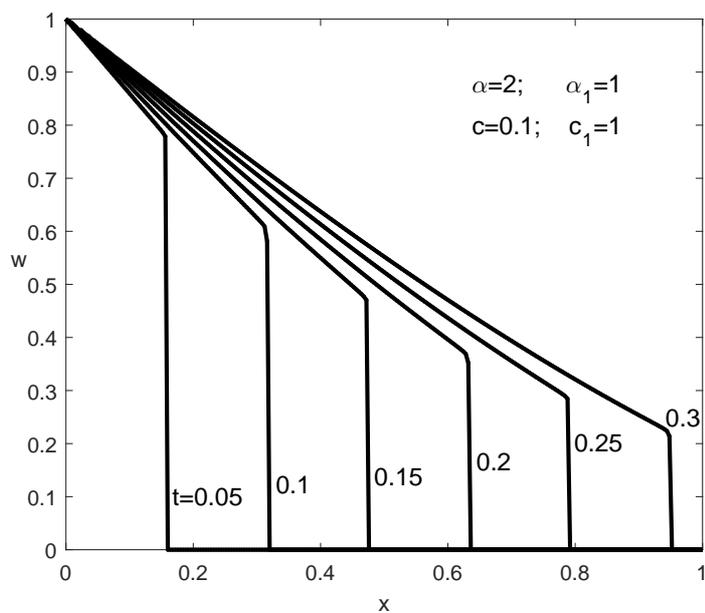}
	\caption{Propagation function $w(x,t)$ for the classical telegraph equation 
	as a function of $x$ 
	for different values of $t$.
	}	
	\end{figure}
	
	\begin{figure}[h]
		\includegraphics[width=0.8\textwidth]{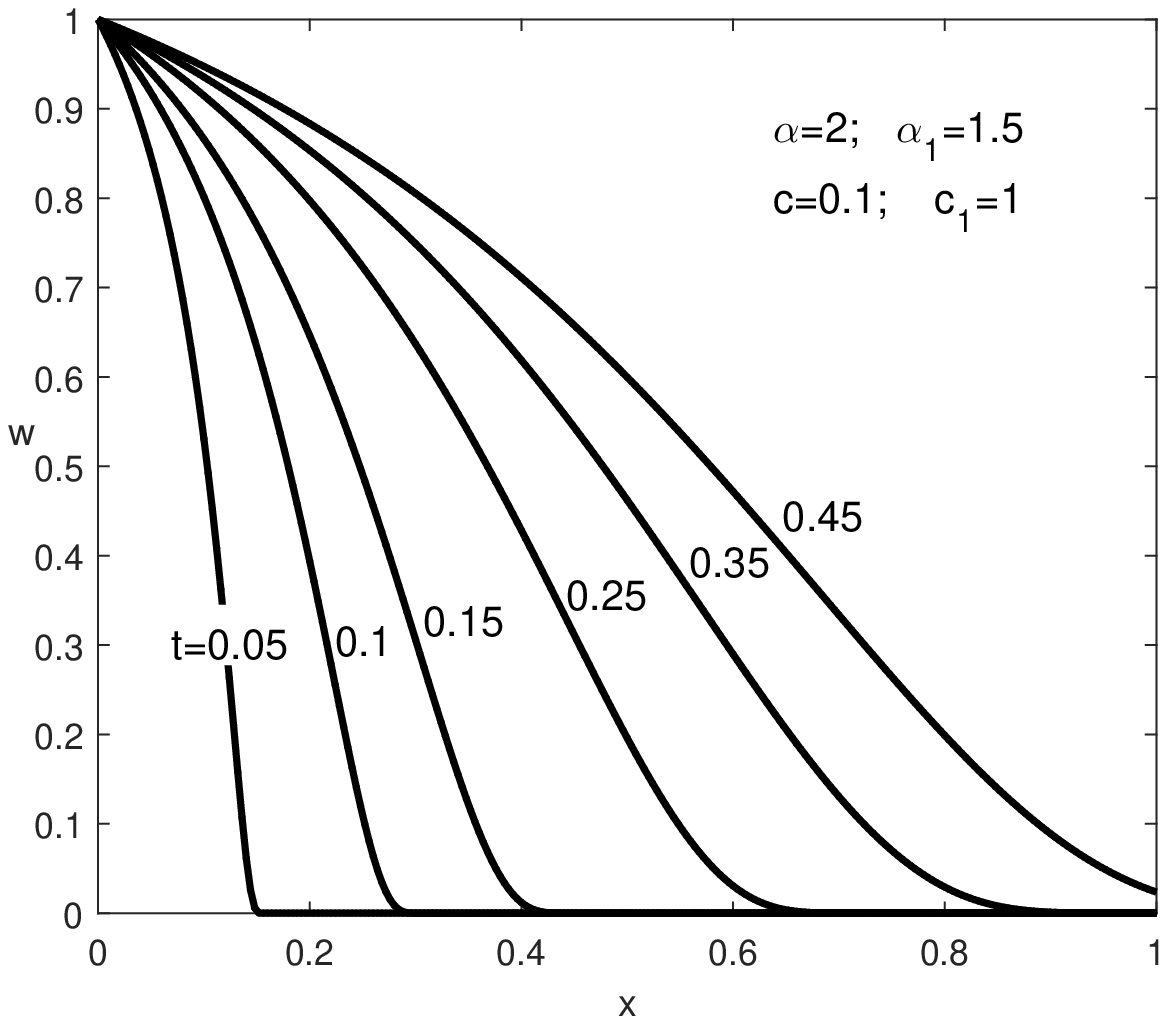}
	\caption{Propagation function $w(x,t)$ for a two-term equation (\ref{A}) with $\al=2$, $\al_1=1.5$, as a function of $x$ 
	for different values of $t$.
	}	
	\end{figure}

	\begin{figure}[h]
		\includegraphics[width=0.8\textwidth]{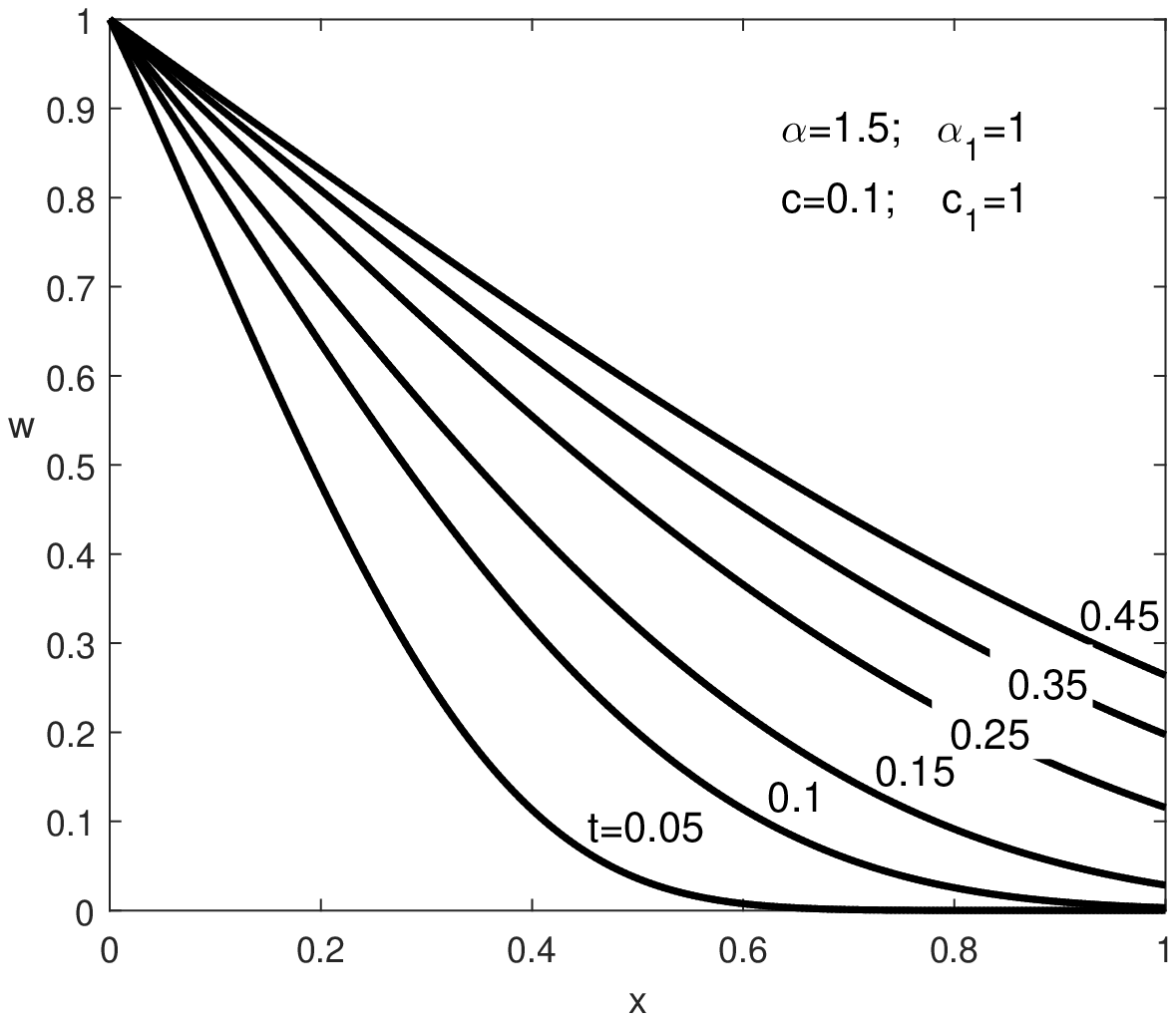}
	\caption{Propagation function $w(x,t)$ for a two-term equation (\ref{A}) with $\al=1.5$, $\al_1=1$, as a function of $x$ 
	for different values of $t$.
	}	
	\end{figure}

\begin{figure}[h]
		\includegraphics[width=0.8\textwidth]{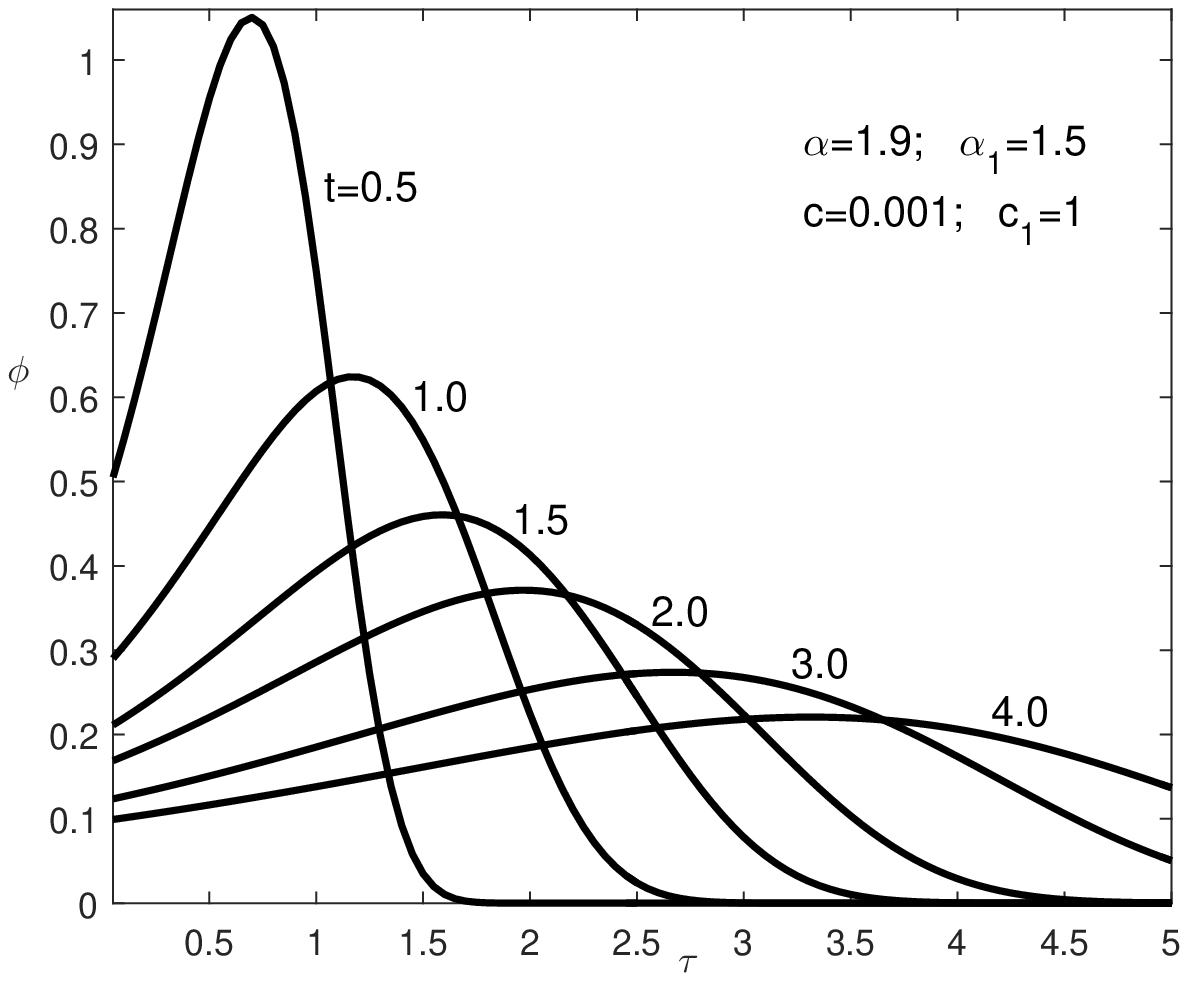}
	\caption{Probability density function $\ph(t,\tau)$ for a two-term equation (\ref{A})  with $\al=1.9$, $\al_1=1.5$, as a function of $\tau$ for different values of $t$.  }	
	\end{figure}

\begin{figure}[h]
		\includegraphics[width=0.8\textwidth]{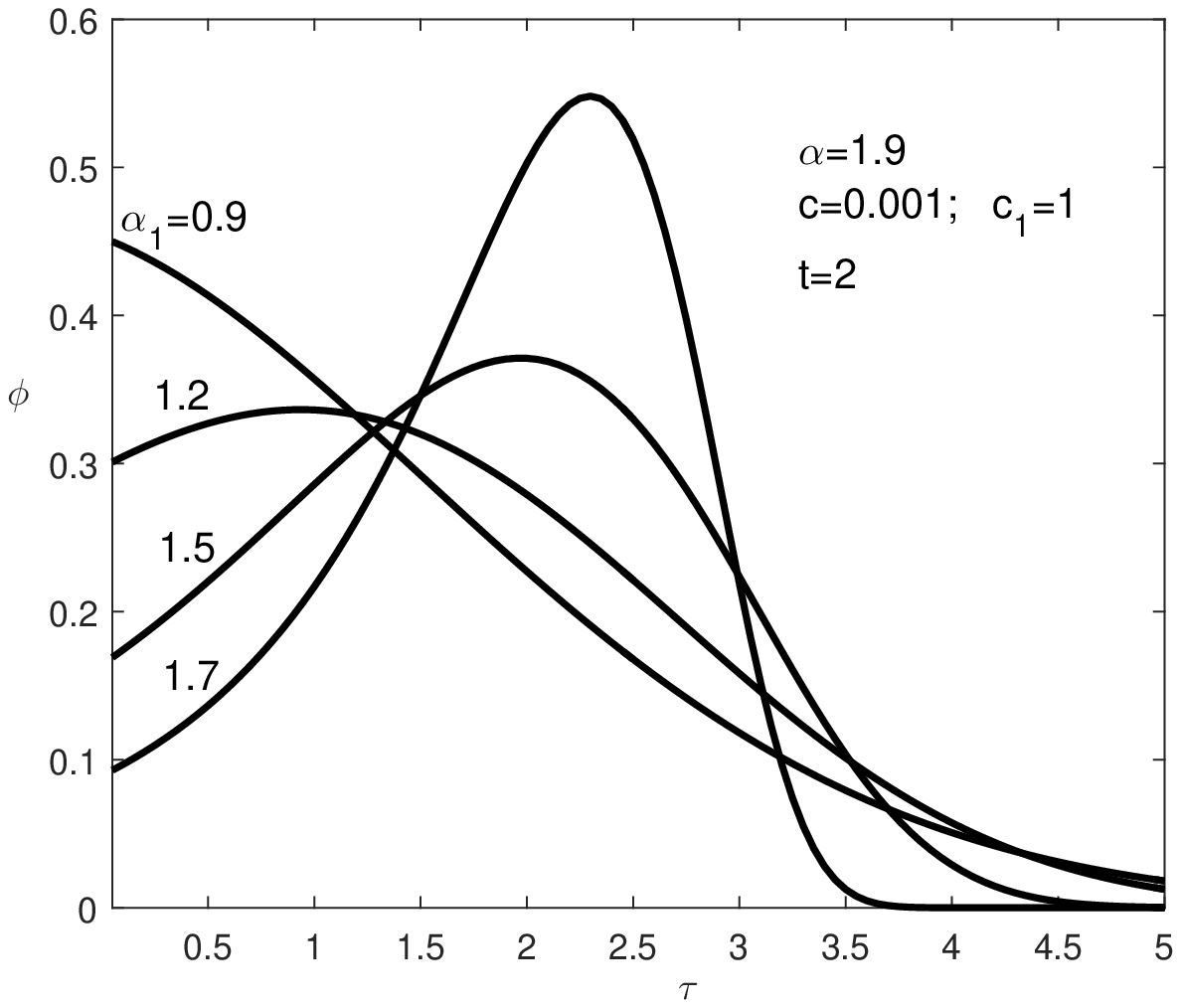}
	\caption{Probability density function $\ph(t,\tau)$ for a two-term equation (\ref{A}) with $\al=1.9$ as a function of $\tau$ for $t=2$ and different values of $\al_1$.  }	
	\end{figure}
	
	\begin{figure}[h]
		\includegraphics[width=0.8\textwidth]{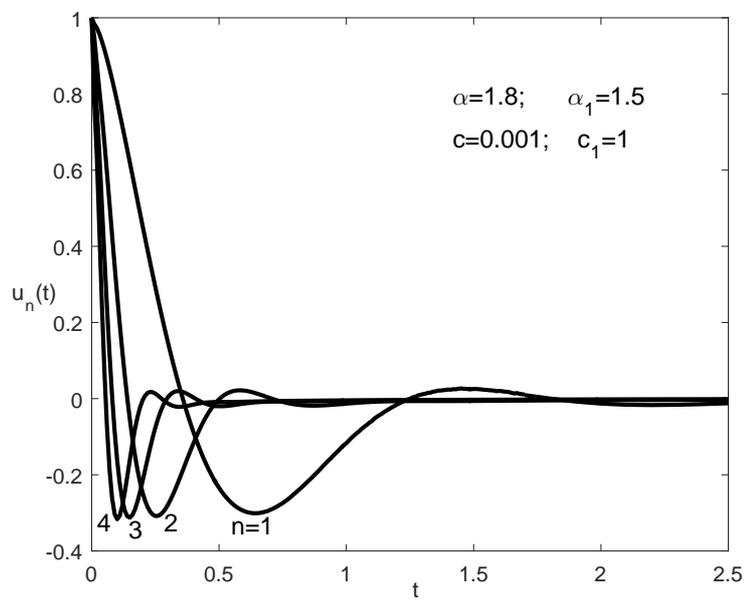}
	\caption{Eigenmodes $u_n(t)$, $n=1,2,3,4,$ for the two-term one-dimensional equation with $\al=1.8$, $\al_1=1.5$.  }	
	\end{figure}

\end{document}